\documentclass[12pt,a4paper]{amsproc}
\usepackage{enumerate,amsmath,amsthm,amssymb,cite,wasysym,mathrsfs,graphicx,caption,subfigure}
\usepackage{float}
\usepackage{color}

\newcommand{\A}{\mathcal{A}}

\DeclareMathOperator{\RE}{Re} \DeclareMathOperator{\IM}{Im}

\numberwithin{equation}{section}
\newtheorem{theorem}{Theorem}[section]
\newtheorem{lemma}[theorem]{Lemma}
\newtheorem{corollary}[theorem]{Corollary}

\theoremstyle{remark}

\allowdisplaybreaks

\makeatletter
\@namedef{subjclassname@2020}{\textup{2020} Mathematics Subject Classification}
\makeatother
\begin{document}
\title[Radius of Starlikeness]{Starlikeness of Analytic Functions\\ with Subordinate Ratios}

\author[R. M. Ali]{Rosihan M. Ali}
\address{School of Mathematical Sciences, Universiti Sains Malaysia,
11800 USM, Penang, Malaysia}
\email{rosihan@usm.my}

\author[K. Sharma]{Kanika Sharma}
\address{Department of Mathematics, Atma Ram Sanatan Dharma College,
University of Delhi, Delhi--110 021, India}
\email{ksharma@arsd.du.ac.in; kanika.divika@gmail.com}

\author{V. Ravichandran}
\address{Department of Mathematics, National Institute of Technology,
Tiruchirappalli--620 015, India} \email{ravic@nitt.edu; vravi68@gmail.com}

\begin{abstract}
Let $h$ be a non-vanishing analytic function in the open unit disc with $h(0)=1$. Consider the class consisting of normalized analytic functions $f$ whose ratios $f(z)/g(z)$, $g(z)/z p(z)$, and $p(z)$ are each subordinate to $h$ for some analytic functions $g$ and $p$. The radius of starlikeness is obtained for this class when $h$ is chosen to be either $h(z)=\sqrt{1+z}$ or $h(z)=e^z$. Further $\mathcal{G}$-radius is also obtained for each of these two classes when $\mathcal{G}$ is a particular widely studied subclass of starlike functions. These include $\mathcal{G}$ consisting of the Janowski starlike functions, and functions which are parabolic starlike.
\end{abstract}

\maketitle
\section{Classes of Analytic Functions}
Let $\mathcal{A}$ denote the class of normalized analytic functions $f(z)=z+\sum_{k=2}^{\infty}a_kz^k$ in the unit disc $\mathbb{D}=\{z\in \mathbb{C}:|z|<1\}$. A prominent subclass of $\mathcal{A}$ is the class $\mathcal{S}^*$ consisting of functions $f \in \mathcal{A}$ such that $f(\mathbb{D})$ is a starlike domain with respect to the origin. Geometrically, this means the linear segment joining the origin to every other point $w \in f(\mathbb{D})$ lies entirely in $f(\mathbb{D})$. Every starlike function in $\mathcal{A}$ is necessarily univalent.

Since $f'(0)$ does not vanish, every function $f\in\mathcal{A}$ is locally univalent at $z=0$. Further, each function $f\in\mathcal{A}$ mirrors the identity mapping near the origin, and thus in particular, maps small circles $|z|=r$ onto curves which bound starlike domains. If $f\in\mathcal{A}$ is also required to be univalent in $\mathbb{D}$, then it is known that $f$ maps the disc $|z| < r$ onto a domain starlike with respect to the origin for every $r \leq r_{0}:=\tanh(\pi/4)$ (see\cite[Corollary, p.\ 98]{duren}). The constant $r_{0}$ cannot be improved. Denoting by $\mathcal{S}$ the class of univalent functions $f\in\mathcal{A}$, the number $r_{0}=\tanh(\pi/4)$ is commonly referred to as the radius of starlikeness for the class $\mathcal{S}$.

Another informative description of the class $\mathcal{S}$ is its radius of convexity. Here it is known that every $f \in \mathcal{S}$ maps the disc $|z| < r$ onto a convex domain for every $r \leq r_{0}:=2 - \sqrt{3}$ \cite[Corollary, p.\ 44]{duren}. Thus the radius of convexity for $\mathcal{S}$ is $r_{0}=2 - \sqrt{3}$.

To formulate a radius description for other entities besides starlikeness and convexity, consider in general two families $\mathcal{G}$ and $\mathcal{M}$ of $\mathcal{A}$. The \emph{$\mathcal{G}$-radius} for the class $\mathcal{M}$, denoted by $R_\mathcal{G}(\mathcal{M})$, is the largest number $R$ such that $r^{-1}f(rz)\in\mathcal{G}$ for every $0< r\leq R$ and $f \in \mathcal{M}$. Thus, for example, an equivalent description of the radius of starlikeness for $\mathcal{S}$ is that the $\mathcal{S}^*$-radius for the class $\mathcal{S}$ is $R_{\mathcal{S}^*}(\mathcal{S}) = \tanh(\pi/4)$.

In this paper, we seek to determine the radius of starlikeness, and certain other $\mathcal{G}$-radius, for particular subclasses $\mathcal{G}$ of $\mathcal{A}$. Several widely-studied subclasses of $\mathcal{A}$ have simple geometric descriptions; these functions are often expressed as a ratio between two functions. Among the very early studies in this direction is the class of close-to-convex functions introduced by Kaplan \cite{kaplan}, and Reade's class \cite{reade}  of close-to-starlike functions. Close-to-convex functions are necessarily univalent, but not so for close-to-starlike functions. Several works, for example those in \cite{ali2,mac1,mac4,mac3,mac2,vravicmft,jain}, have advanced studies in classes of functions characterized by the ratio between functions $f$ and $g$ belonging to given subclasses of $\mathcal{A}$.

 In this paper, we examine two different subclasses of functions in $\mathcal{A}$ satisfying a certain subordination link of ratios. Interestingly, these classes contain non-univalent functions. An analytic function $f$ is \emph{subordinate} to an analytic function $g$, written $f\prec g$, if
\[f(z)=g(w(z)), \quad \quad z \in \mathbb{D}, \]
for some analytic self-map $w$ in $\mathbb{D}$ with $|w(z)| \leq |z|$. The function $w$ is often referred to as a Schwarz function.

Now let $h$ be a non-vanishing analytic function in $\mathbb{D}$ with $h(0)=1$. The classes treated in this paper consist of functions $f\in \mathcal{A}$ whose ratios $f(z)/g(z)$, $g(z)/z p(z)$, and $p(z)$, are each subordinate to $h$ for some analytic functions $g$ and $p$:
\[\frac{f(z)}{g(z)} \prec h(z), \quad \frac{g(z)}{z p(z)} \prec h(z), \quad p(z)\prec h(z).\]
When $p$ is the constant one function, then the class contains functions $f \in \mathcal{A}$ satisfying the subordination of ratios \[\frac{f(z)}{g(z)} \prec h(z), \quad \frac{g(z)}{z} \prec h(z).\] For $h(z)=(1+z)/(1-z),$ and other appropriate choices of $h$, these functions have earlier been studied, notably by MacGregor in \cite{mac1,mac4,mac3,mac2}, and Ratti in \cite{ratti68,ratti70}. Recent investigations include those in \cite{ali2,vravicmft,jain}.

In this paper, two specific choices of the function $h$ are made: $h(z) = \sqrt{1+z}$, and $h(z) = e^{z}$.

\subsection*{The Class $\mathcal{T}_{1}$} This is the class given by
\[ \mathcal{T}_{1}:=\left\{f\in \mathcal{A}: \frac{f(z)}{g(z)} \prec \sqrt{1+z}, \frac{g(z)}{z p(z)} \prec \sqrt{1+z} \quad \text{for some} \quad g \in \mathcal{A}, \, p(z) \prec \sqrt{1+z}  \right\}.\]
This class is non-empty: let $f_{1}, g_{1}, p_{1} : \mathbb{D} \to \mathbb{C}$ be given by
\[ f_{1}(z)= z(1+z)^{3/2}, \quad g_{1}(z)=z(1+z) \quad \text{and} \quad p_{1}(z)=\sqrt{1+z}.\]
Then $f_{1}(z)/g_{1}(z)\prec \sqrt{1+z}$ and $g_{1}(z)/zp_{1}(z)\prec \sqrt{1+z}$, so that $f_{1}\in \mathcal{T}_{1}.$ The function $f_{1}$ will be shown to play the role of an extremal function for the class $\mathcal{T}_{1}$. Since  $f_{1}'$ vanishes at $z=-2/5$,  the function $f_{1}$ is non-univalent, and thus, the class $\mathcal{T}_{1}$ contains non-univalent functions. Incidentally, $f_{1}$ demonstrates the radius of univalence for $\mathcal{T}_{1}$ is at most $2/5$. In Theorem~\ref{p8thm3.1}, the radius of starlikeness for $\mathcal{T}_{1}$ is shown to be $2/5$, whence $\mathcal{T}_{1}$ has radius of univalence $2/5$.

The following is a useful result in investigating the starlikeness of the class $\mathcal{T}_{1}$.

\begin{lemma} Let $p(z)\prec \sqrt{1+z}$. Then {p} satisfies the sharp inequalities
\begin{equation}\label{p8eqn2.2a}  \sqrt{1-r} \leq |p(z)|\leq \sqrt{1+r}, \quad |z|\leq r,
\end{equation}  and
\begin{equation}\label{p8eqn2.2}
\left|\frac{z p'(z)}{p(z)}\right| \leq \frac{r}{2(1-r)}, \quad |z|\leq r.
\end{equation}
\end{lemma}

\begin{proof}
If $p(z) \prec \sqrt{1+z}$, then $p^{2}(z)=1+w(z)$ for some Schwarz function $w$. The well-known Schwarz lemma shows that  $|w(z)|\leq |z|$ and  \begin{equation}\label{p8eqn2.1}
|w'(z)| \leq \frac{1-|w(z)|^{2}}{1-|z|^{2}}.
\end{equation}
Therefore,
\[|p(z)|^{2}=|1+w(z)|\leq 1+|w(z)|\leq 1+|z|\leq 1+r\]
for $|z|\leq r$, that is, $|p(z)|\leq \sqrt{1+r}$ for $|z|\leq r$. Similarly, $|p(z)|\geq \sqrt{1-r} $ for $|z|\leq r $.

Since  $2zp'(z)/p(z)=zw'(z)/(1+w(z))$, the inequality ~\eqref{p8eqn2.1} readily shows
\[
2\left|\frac{z p'(z)}{p(z)}\right| \leq \frac{|z||w'(z)|}{1-|w(z)|}\leq \frac{|z|(1+|w(z)|)}{1-|z|^{2}}\leq \frac{|z|(1+|z|)}{1-|z|^{2}}= \frac{|z|}{1-|z|}\leq \frac{r}{1-r}\]
for $|z|\leq r $. This proves \eqref{p8eqn2.2}. The  inequalities are sharp for the function $p(z)=\sqrt{1+z} \qedhere$.
\end{proof}

For $f \in \mathcal{T}_{1}$, let $p_{1}(z)=f(z)/g(z)$ and $p_{2}(z)=g(z)/z p(z)$. Then $f(z)=zp(z)p_{1}(z)p_{2}(z)$ and
\begin{equation}\label{p8eqn2.3}
\left|\frac{z f'(z)}{f(z)}-1\right| \leq \left|\frac{z p'(z)}{p(z)}\right| + \left|\frac{z p_{1}'(z)}{p_{1}(z)}\right| + \left|\frac{z p_{2}'(z)}{p_{2}(z)}\right|.
\end{equation}
Since $p, p_{1}, p_{2} \prec \sqrt{1+z}$, we deduce from ~\eqref{p8eqn2.2} and ~\eqref{p8eqn2.3} that
\begin{equation}\label{p8eqn2.4}
\left|\frac{z f'(z)}{f(z)}-1\right| \leq \frac{3 r}{2(1-r)}, \quad |z| \leq r,
\end{equation}
for each function $f \in \mathcal{T}_{1}$. Sharp growth inequalities also follow from \eqref{p8eqn2.2a}:
\[r(1-r)^{3/2} \leq |f(z)|\leq r(1+r)^{3/2}  \]
for each $f \in \mathcal{T}_{1}$. Crude distortion inequalities can readily be obtained from \eqref{p8eqn2.4} and the growth inequality; however, finding sharp estimates remain an open problem.

\subsection*{The Class $\mathcal{T}_{2}$} This class is defined by \[ \mathcal{T}_{2}:=\left\{f\in \mathcal{A}: \frac{f(z)}{g(z)} \prec e^{z}, \frac{g(z)}{z p(z)} \prec e^{z} \quad \text{for some} \quad g \in \mathcal{A},\  p(z) \prec e^{z}  \right\}.\]
Let $f_{2}, g_{2}, p_{2} : \mathbb{D} \to \mathbb{C}$ be given by
\[ f_{2}(z)= ze^{3z}, \quad g_{2}(z)=ze^{2z} \quad \text{and} \quad p_{2}(z)=e^{z}.\]
Evidently, $f_{2}(z)/g_{2}(z)\prec e^{z}$, $g_{2}(z)/zp_{2}(z)\prec e^{z}$ so that $f_{2}\in \mathcal{T}_{2}$, and the class $\mathcal{T}_{2}$ is non-empty.
 Similar to $f_{1} \in \mathcal{T}_{1}$, the function $f_{2}$ plays the role of an extremal function for the class $\mathcal{T}_{2}$. The Taylor series expansion for $f_{2}$  is
\[ f_{2}(z)=z+3 z^2+\frac{9 z^3}{2}+\frac{9 z^4}{2}+\frac{27 z^5}{8}+\dotsb.\]
Comparing the second  coefficient, it is clear that $f_{2}$ is non-univalent. Hence the class $\mathcal{T}_{2}$ contains non-univalent functions. The derivative $f_{2}'$ vanishes at $z=-1/3$, which shows the radius of univalence for $\mathcal{T}_{2}$ is at most $1/3$. From  Theorem ~\ref{p8thm3.1}, the radius of starlikeness  is shown to be $1/3$, and so the radius of univalence for $\mathcal{T}_{2}$ is $1/3$.

%

\begin{lemma} Every $p(z)\prec e^z$ satisfies the sharp inequalities
\begin{equation}\label{p8eqn25a}  e^{-r} \leq |p(z)|\leq e^r, \quad |z|\leq r,
\end{equation}  and
\begin{equation}\label{p8eqn25}
\left|\frac{z p'(z)}{p(z)}\right|\leq  \begin{cases}  r, &  |z|\leq r \leq \sqrt{2}-1 \\ \dfrac{(1+r^{2})^{2}}{4(1-r^{2})}, & |z|=r \geq \sqrt{2}-1.\end{cases}
\end{equation}
\end{lemma}

\begin{proof}
 Let $p(z) \prec e^{z}$. Since $p(z)=e^{w(z)}$ for some Schwarz self-map $w$ satisfying $|w(z)|\leq |z|$, it follows that
 \[|p(z)|=e^{\RE w(z)} \leq e^{|w(z)|} \leq e^{|z|}.\] The function $w$ also satisfy the sharp inequality (see \cite[Corollary, p.\ 199]{duren})
\begin{equation}\label{p8eqn2.5a}
|w'(z)|\leq  \begin{cases}  1, & r=|z|\leq \sqrt{2}-1 \\ \dfrac{(1+r^{2})^{2}}{4r(1-r^{2})}, & r \geq \sqrt{2}-1.\end{cases}
\end{equation}
From $z p'(z)/p(z)=z w'(z)$, we conclude that
\begin{equation*}
\left|\frac{z p'(z)}{p(z)}\right|\leq  \begin{cases}  r, & r=|z|\leq \sqrt{2}-1 \\ \dfrac{ (1+r^{2})^{2}}{4(1-r^{2})}, & r \geq \sqrt{2}-1.\end{cases}
\end{equation*} This inequality is sharp for $p(z)=e^z$ and $r=|z|\leq \sqrt{2}-1$. It is also sharp in the remaining interval for the function $p(z)=e^{w(z)}$, where $w$ is the extremal function for which equality holds in \eqref{p8eqn2.5a}.
\end{proof}

For $f \in \mathcal{T}_{2}$, let $p_{1}(z)=f(z)/g(z)$ and $p_{2}(z)=g(z)/z p(z)$. Then $f(z)=zp(z)p_{1}(z)p_{2}(z)$ and
\begin{equation}\label{p8eqn2.6}
\left|\frac{z f'(z)}{f(z)}-1\right| \leq \left|\frac{z p'(z)}{p(z)}\right| + \left|\frac{z p_{1}'(z)}{p_{1}(z)}\right| + \left|\frac{z p_{2}'(z)}{p_{2}(z)}\right|.
\end{equation}
Since $p, p_{1}, p_{2} \prec e^{z}$, estimates ~\eqref{p8eqn25} and ~\eqref{p8eqn2.6} show that
\begin{equation}\label{p8eqn2.7}
\left|\frac{z f'(z)}{f(z)}-1\right| \leq \begin{cases}  3r, & r=|z|\leq \sqrt{2}-1 \\ \dfrac{3(1+r^{2})^{2}}{4(1-r^{2})}, & r \geq \sqrt{2}-1.\end{cases}
\end{equation}
for each function $f \in \mathcal{T}_{2}$. It also follows from \eqref{p8eqn25a} that
\[re^{-3r} \leq |f(z)|\leq re^{3r}  \] holds  for each function $f \in \mathcal{T}_{2}$, and that these estimates are sharp.

In this paper, we shall adopt the commonly used notations for subclasses of $\mathcal{A}$. First, for $0\leq \alpha<1$, let $\mathcal{S}^*(\alpha)$ denote the class of starlike functions of order $\alpha$ consisting of functions $f\in\mathcal{A}$ satisfying the subordination
\[\frac{zf'(z)}{f(z)} \prec \frac {1+(1-2\alpha)z}{1-z}.\]
Thus
\[\RE \frac{z f'(z)}{f(z)} > \alpha, \quad \quad z \in \mathbb{D}. \]
The case $\alpha=0$ corresponds to the classical functions whose image domains are starlike with respect to the origin. Various other starlike subclasses of  $\mathcal{A}$ occurring in the literature can be expressed in terms of the subordination
\begin{equation}\label{rma7}
\frac{zf'(z)}{f(z)} \prec \varphi (z)
\end{equation}
for suitable choices of the superordinate function $\varphi$. When $\varphi:\mathbb{D}\to \mathbb{C}$ is chosen to be  $\varphi (z):=(1+Az)/(1+Bz)$, $-1\leq B < A\leq 1$, the subclass derived is denoted by $\mathcal{S}^*[A,B]$. Functions $f \in \mathcal{S}^*[A,B]$ are known as Janowski starlike. When $\varphi (z):=1+(2/\pi^{2})( \left(\log ((1+\sqrt{z})/(1-\sqrt{z}))\right)^{2})$, the subclass is denoted by $\mathcal{S}^{*}_{p}$, and its functions are called parabolic starlike.

In Section 2 of this paper, the radius of starlikeness, Janowski starlikeness, and parabolic starlikeness are found for the classes $\mathcal{T}_{i}$, with $i=1,2$. Section 3 deals with the determination of the $\mathcal{G}$-radius for the class $\mathcal{T}_{i}$ with $i=1,2$, for certain other subclasses $\mathcal{G}$ occurring in the literature. These classes are associated with particular choices of the superordinate function $\varphi$ in \eqref{rma7}. As mentioned earlier, the $\mathcal{G}$-radius for a given class $\mathcal{M}$, denoted by $R_\mathcal{G}(\mathcal{M})$, is the largest number $R$ such that $r^{-1}f(rz)\in\mathcal{G}$ for every $0< r\leq R$ and $f \in \mathcal{M}$. It will become apparent in the forthcoming proofs that there are common features in the methodology of finding the $\mathcal{G}$-radius for each of these subclasses.

\section{Starlikeness of order $\alpha$, Janowski and parabolic starlikeness}\label{p8sec3}

 The first result deals with the $\mathcal{S}^*(\alpha)$-radius (radius of starlikeness of order $\alpha$) for the classes $\mathcal{T}_{1}$ and $\mathcal{T}_{2}$. This radius is shown to equal the $\mathcal{S}^*_\alpha$-radius, where $\mathcal{S}^*_\alpha$ is the subclass containing functions $f\in\mathcal{A}$ satisfying $|zf'(z)/f(z)-1|<1-\alpha$. The latter condition also implies that $\mathcal{S}^*_\alpha \subset \mathcal{S}^*(\alpha)$.

\begin{theorem}\label{p8thm3.1}
Let $ 0\leq \alpha <1$. The radius of starlikeness of order $\alpha$ for $\mathcal{T}_{1}$ and $\mathcal{T}_{2}$ are
\begin{enumerate}[(i)]
	\item $R_{\mathcal{S}^*(\alpha)}(\mathcal{T}_{1})= R_{\mathcal{S}^*_\alpha}(\mathcal{T}_{1})=2(1-\alpha)/(5-2\alpha)$,

	\item $R_{\mathcal{S}^*(\alpha)}(\mathcal{T}_{2})=R_{\mathcal{S}^*_\alpha}(\mathcal{T}_{2})=(1-\alpha)/3$.
\end{enumerate}
\end{theorem}

\begin{proof}
$(i)$ The function $\sigma(r)=(2-5r)/(2-2r)$ is a decreasing function on $[0,1)$. Further, the number  $R_1:=2(1-\alpha)/(5-2\alpha)$ is  the root of the equation $\sigma(r)=\alpha$.  For $f \in \mathcal{T}_{1}$ and $0 < r \leq R_1$, the inequality \eqref{p8eqn2.4} readily yields
\[\RE \frac{z f'(z)}{f(z)} \geq 1-\frac{3 r}{2(1-r)} = \frac{2-5r}{2-2r} =\sigma(r)\geq  \sigma(R_1)=\alpha\] and
\[
\left|\frac{z f'(z)}{f(z)}-1\right| \leq \frac{3 r}{2(1-r)}=1-\sigma(r) \leq 1-\sigma(R_1) =1-\alpha. \]

At $z=-R_1$, the function $f_1\in \mathcal{T}_{1}$ given by  $f_{1}(z)=z(1+z)^{3/2}$ yields
\[  \frac{z f'_{1}(z)}{f_{1}(z)}= \frac{2+5z}{2+2z}=\frac{2-5R_1}{2-2R_1}=\alpha. \]
Thus
\[ \RE \frac{z f'_{1}(z)}{f_{1}(z)}= \alpha \quad \text{and}\quad \left|\frac{z f'_{1}(z)}{f_{1}(z)}-1\right|=1-\alpha. \]
This proves that the $ \mathcal{S}^*(\alpha) $ and $ \mathcal{S}^*_\alpha $ radii for $\mathcal{T}_1$ are the same number $R_1$.

$(ii)$  Consider $\omega(r)=1-3r$, \, $0\leq r <1$.  The number  $R_2=(1-\alpha)/3<1/3$ is clearly the root of the equation $\omega(r)=\alpha$. Since $\omega$ is decreasing, then $\omega(r) \geq \omega(R_2)=\alpha$ for each $f \in \mathcal{T}_{2}$ and $0 < r \leq R_2$. It follows from ~\eqref{p8eqn2.7} that
\[\RE \frac{z f'(z)}{f(z)} \geq 1-3r =\omega(r)\geq \alpha,\]
and
\[\left| \frac{z f'(z)}{f(z)}-1\right| \leq 3r =1-\omega(r)\leq 1- \alpha.\]

Evaluating the function $f_{2}(z)=ze^{3z}$ at $z=-R_2$ yields
\[   \frac{z f'_{2}(z)}{f_{2}(z)}= 1-3R=\alpha.  \]
Hence
\[ \RE \frac{z f'_{2}(z)}{f_{2}(z)}= \alpha  \quad \text{and}\quad \left|\frac{z f'_{2}(z)}{f_{2}(z)}-1\right|=1-\alpha. \]
This proves that the $ \mathcal{S}^*(\alpha) $ and $ \mathcal{S}^*_\alpha $ radii for the class $\mathcal{T}_2$ are the same number $R_2$.
\end{proof}

Next we find the $\mathcal{S}^*[A,B]$-radius (Janowski starlikeness) for $\mathcal{T}_{1}$ and $\mathcal{T}_{2}$. Recall that $\mathcal{S}^*[A,B]$ consists of analytic functions $f\in \mathcal{A}$ satisfying the subordination $zf'(z)/f(z)\prec (1+Az)/(1+Bz)$, $-1\leq B < A\leq 1$.
\begin{theorem}\label{p8thm3.1a}
\begin{enumerate}[(i)]
 \item Every $f \in \mathcal{T}_{1}$ is Janowski starlike in the disc $\mathbb{D}_r=\{z:|z|<r\}$ for $r \leq 2(A-B)/(3(1+|B|)+2(A-B))$. If $B < 0$, then $R_{\mathcal{S}^*[A,B]}(\mathcal{T}_{1})= 2(A-B)/(3+2A-5B))$.

 \item The radius of Janowski starlikeness for $\mathcal{T}_{2}$ is $R_{\mathcal{S}^*[A,B]}(\mathcal{T}_{2})=(A-B)/(3(1-B))$.


\end{enumerate}
\end{theorem}

\begin{proof} Since $\mathcal{S}^*[A,-1]=\mathcal{S}^*((1-A)/2)$, the results in the case $B=-1$ follow from Theorem~\ref{p8thm3.1}. We now prove the results when  $-1<B<A\leq 1$.

$(i)$ Let $f\in\mathcal{T}_1$ and write $w=zf'(z)/f(z)$. Then ~\eqref{p8eqn2.4} shows that $|w-1|\leq 3r/(2(1-r))$ for $|z|\leq r$. For $0\leq r\leq R_1:= 2(A-B)/(3(1+|B|)+2(A-B))$, then  $3R_1/((2(1-R_1)) =(A-B)/(1+|B|)$.

For $0\leq r\leq R_1$, we first show that the disc
\[\left\{w : |w-1|\leq \dfrac{3R_1}{2(1-R_1)}=\dfrac{A-B}{1+|B|}\right\}\]
is contained in the images  of the unit disc under the mapping $(1+Az)/(1+Bz)$. As $B\neq -1$, the image is the disc given
by
\[\left\{w:\left|w- \dfrac{1-AB}{1-B^2}\right| < \dfrac{A-B}{1-B^2}\right\}.\]
Silverman \cite[p. 50-51]{silver} has shown that  the disc
\[\{w:|w-c|<d\}\subset \{w: |w-a|< b\}\]
if and only if $|a-c| \leq b-d$. With the choices $c=1$, $d=(A-B)/(1+|B|)$, $a=(1-AB)/(1-B^2)$ and  $b=(A-B)/(1-B^2)$, then $|a-c|=|B|(A-B)/(1-B^2)= b-d$. This proves that $\mathcal{S}^*[A,B]$ radius is at least $R_1$.

To prove sharpness,
consider the function $f_1\in\mathcal{T}_1$ given by $f_1(z)=z(1+z)^{3/2}$. Evidently, $zf_1'(z)/f_1(z)=(2+5z)/(2+2z)$.  For $B<0$, evaluating at $z=-R_1$, then $zf_1'(z)/f_1(z)=1+3z/(2+2z)=1-(A-B)/(1+|B|)=(1-A)/(1-B)$. This shows that
\[\left|\frac{zf_1'(z)}{f_1(z)} - \frac{1-AB}{1-B^2}\right| =\left|\frac{1-A}{1-B} - \frac{1-AB}{1-B^2}\right|= \frac{A-B}{1-B^2},  \]
proving sharpness in the case $B<0$.


$(ii)$  Let $f\in\mathcal{T}_2$ and $w:=zf'(z)/f(z)$. It follows from ~\eqref{p8eqn2.7} that $|w-1|\leq 3r$ for $|z|\leq r$. For $0\leq r\leq R_2:= (A-B)/(3(1+|B|))$, we see that the disc $\{w : |w-1|\leq 3R_2=(A-B)/(1+|B|)\}$ is contained in the disc $\{w:|w-(1-AB)/(1-B^2)| < (A-B)/(1-B^2)\}$, as in the proof of (i). This proves that $\mathcal{S}^*[A,B]$ radius is at least $R_2$.   The result is sharp for the function $f_2\in\mathcal{T}_2$ given by  the function $f_2(z)=ze^{3z}$.
\end{proof}

%

The function $\varphi_{PAR}:\mathbb{D}\to\mathbb{C}$   given by
\begin{equation*}
\varphi_{PAR}(z):=1 + \frac{2}{\pi^{2}} \left(\log \frac{1+\sqrt{z}}{1-\sqrt{z}}\right)^{2},\quad \IM\sqrt{z}\geq 0,
\end{equation*}
maps $\mathbb{D}$ into the parabolic region
\[\varphi_{PAR}(\mathbb{D})=\left\{w=u+iv:v^{2}<2u-1\right\}=\left\{w:\RE w>|w-1| \right\}.\]
The class $\mathcal{C}(\varphi_{PAR})=\{f\in\mathcal{A}: 1+zf''(z)/f'(z)\prec \varphi_{PAR}(z)\}$ is the class of uniformly convex functions introduced by Goodman \cite{goodman}. The corresponding class $\mathcal{S}^*_{p}:=\mathcal{S}^*(\varphi_{PAR})=\{f\in\mathcal{A}:  zf'(z)/f(z)\prec \varphi_{PAR}(z)\}$ introduced by R{\o}nning \cite{ron} is known as the class of  parabolic starlike functions. The class  $\mathcal{S}^*_{p}$ consists of functions $f\in\mathcal{A}$ satisfying
\[\RE\left(\frac{z f'(z)}{f(z)}\right)>\left|\frac{z f'(z)}{f(z)}-1\right|,\quad z\in\mathbb{D}. \]
Evidently, every parabolic starlike function is also starlike of order 1/2. The radius of parabolic starlikeness for the class $\mathcal{T}_{1}$ and $\mathcal{T}_{2}$ is given in the next result.

\begin{corollary}\label{p8thm3.2}
The radius of parabolic starlikeness for $\mathcal{T}_{1}$ and $\mathcal{T}_{2}$ is respectively equal to its radius of starlikeness of order $1/2$. Thus,
\begin{enumerate}[(i)]
		\item $R_{\mathcal{S}^*_{p}}(\mathcal{T}_{1})=1/4$,
		\item $R_{\mathcal{S}^*_{p}}(\mathcal{T}_{2})=1/6$.
	\end{enumerate}
\end{corollary}

\begin{proof} Shanmugam and Ravichandran \cite[p.\ 321]{vravicmft} proved that
\begin{equation*}\label{p8eqn3.2}
\{w:|w-a|<a-1/2\}\subseteq \{w:\RE w>|w-1|\}
\end{equation*}
for $1/2<a \leq3/2$. Choosing $a=1$, this implies that $\mathcal{S}^*_{1/2} \subset \mathcal{S}^*_{p}$. Every parabolic starlike function is also starlike of order 1/2, whence the inclusion $\mathcal{S}^*_{1/2} \subset \mathcal{S}^*_{p}\subset \mathcal{S}^*(1/2)$. Therefore, for any class $\mathcal{F}$, readily
$R_{\mathcal{S}^*_{1/2}}(\mathcal{F})\leq R_{\mathcal{S}^*_{p}}(\mathcal{F})\leq R_{\mathcal{S}^*(1/2)}(\mathcal{F})$.

When $\mathcal{F}=\mathcal{T}_i$,  $i=1,2$, Theorem~\ref{p8thm3.1} gives
$R_{\mathcal{S}^*(\alpha)}(\mathcal{T}_{i})=R_{\mathcal{S}^*_\alpha}(\mathcal{T}_{i})$.
This shows that
$R_{\mathcal{S}^*_{1/2}}(\mathcal{T}_{i}) =R_{\mathcal{S}^*_{p}}(\mathcal{T}_{i}) = R_{\mathcal{S}^*(1/2)}(\mathcal{T}_{i})$. Since $  R_{\mathcal{S}^*(1/2)}(\mathcal{T}_{1}) =1/4$ and $R_{\mathcal{S}^*(1/2)}(\mathcal{T}_{2})= 1/6$ from Theorem~\ref{p8thm3.1}, it follows that $R_{\mathcal{S}^*_{p}}(\mathcal{T}_{1})=1/4$ and $R_{\mathcal{S}^*_{p}}(\mathcal{T}_{2})=1/6$.
\end{proof}

\section{Further radius of starlikeness}

In this section, we find the $\mathcal{G}$-radius for the class $\mathcal{T}_{i}$ with $i=1,2$, for certain other widely studied subclasses $\mathcal{G}$. These are associated with particular choices of the superordinate function $\varphi$ in \eqref{rma7}.

Denote by $\mathcal{S}^*_{\exp}:= \mathcal{S}^*(e^{z})$ the class associated with $\varphi (z):=e^{z}$ in \eqref{rma7}. This class was introduced by Mendiratta \textit{et al.} \cite{men1}, and it consists of functions $f\in\A$ satisfying the condition $|\log(zf'(z)/f(z))|<1$.
The following result gives the radius of exponential starlikeness for the classes $\mathcal{T}_{1}$ and $\mathcal{T}_{2}$.

\begin{corollary}\label{p8thm3.4}
The $\mathcal{S}^*_{\exp}$-radius for the class $\mathcal{T}_{1}$ is
\[R_{\mathcal{S}^*_{\exp}}(\mathcal{T}_{1})= (2-2e)/(2-5e)\approx 0.296475,\]
while that of $\mathcal{T}_{2}$ is
\[R_{\mathcal{S}^*_{\exp}}(\mathcal{T}_{2})=(e-1)/3e.\]
\end{corollary}

\begin{proof} Mendiratta \textit{et al.} \cite[Lemma 2.2]{men1} proved that \[\{w:|w-a|<a-1/e\}\subseteq \{w:|\log w|<1\} \] for $e^{-1} \leq a \leq (e+e^{-1})/2$, and this inclusion with $a=1$ gives $\mathcal{S}^* _{1/e}\subset \mathcal{S}^*_{\exp}$. It was also shown in \cite[Theorem 2.1 (i)]{men1}  that   $\mathcal{S}^*_{\exp}\subset \mathcal{S}^* (1/e)$. Therefore, $\mathcal{S}^* _{1/e}\subset \mathcal{S}^*_{\exp}\subset \mathcal{S}^* (1/e),$  which, as a consequence of  Theorem~\ref{p8thm3.1}, established the result.
\end{proof}

%

Corollary ~\ref{p8thm3.3} investigates the radius of cardioid starlikeness for each class $\mathcal{T}_{1}$ and $\mathcal{T}_{2}$. The class $S^*_C:=\mathcal{S}^*(\varphi_{CAR})$, where $\varphi_{CAR}(z)=1+4z/3+2z^{2}/3$ in \eqref{rma7}, was introduced and studied in \cite{jain,sharma,sharma2, sharma3}. Descriptively, $f\in S^*_C$ provided $zf'(z)/f(z)$ lies in the region bounded by the cardioid $\Omega_C :=\{w=u+iv: (9 u^2+9 v^2-18u+5)^2- 16 (9 u^2+9 v^2- 6u+1)=0\}$.

\begin{corollary}\label{p8thm3.3}
	The following are the $S^*_C$-radius for the classes $\mathcal{T}_{1}$ and $\mathcal{T}_{2}$:
	\begin{enumerate}[(i)]
		\item $R_{S^*_C}(\mathcal{T}_{1})= 4/13$,
		\item $R_{S^*_C}(\mathcal{T}_{2})=2/9$.
	\end{enumerate}
\end{corollary}

\begin{proof}  Sharma \textit{et al.}\cite{jain} proved that  $\{w:|w-a|<a-1/3\}\subseteq \Omega_C$ for $1/3<a\leq 5/3$, and this inclusion with $a=1$ gives $\mathcal{S}^* _{1/3}\subset \mathcal{S}^*_{C}$. Thus $R_{\mathcal{S}^* _{1/3}}(\mathcal{T}_i)\leq R_{\mathcal{S}^*_{C}}(\mathcal{T}_i) $ for $i=1,2$. To complete the proof, we demonstrate $ R_{\mathcal{S}^*_{C}}(\mathcal{T}_i)\leq R_{\mathcal{S}^* _{1/3}}(\mathcal{T}_i) $ for $i=1,2$.

$(i)$ Evaluating the function $f_{1}(z)=z(1+z)^{3/2}$ at $z=-R=-R_{S^*_{1/3}}(\mathcal{T}_{1})=-4/13$ gives
	\[  \frac{z f'_{1}(z)}{f_{1}(z)}=\frac{2+5z}{2+2z}=\frac{2-5R}{2-2R}= \frac{1}{3}=\varphi_{CAR}(-1).\]
	Thus, $R_{S^*_C}(\mathcal{T}_{1})\leq 4/13$.
	
	$(ii)$ Similarly, at $z=-R=-R_{S^*_{1/3}}(\mathcal{T}_{2})=-2/9$, the function $f_{2}(z)=ze^{3z}$ yields
\[  \frac{z f'_{2}(z)}{f_{2}(z)}=1+3z=1-3R =\frac{1}{3}=\varphi_{CAR}(-1). \]
This proves that  $R_{S^*_C}(\mathcal{T}_{2})\leq 2/9.$
\end{proof}	

In 2019, Cho \textit{et al.}\cite{viren} studied the class $\mathcal{S}^*_{\text{sin}}:=\mathcal{S}^*(1+\sin z)$ consisting of functions $f\in\A$ satisfying the condition $zf'(z)/f(z) \prec 1+ \sin z$. We find the $\mathcal{S}^*_{\text{sin}}$-radius for the classes $\mathcal{T}_{1}$ and $\mathcal{T}_{2}$.

\begin{corollary}\label{p8thm3.5}
	The following are the  $ \mathcal{S}_{\text{sin}}^{*}$-radius for each class $\mathcal{T}_{1}$ and $\mathcal{T}_{2}$:
	\begin{enumerate}[(i)]
		\item $R_{\mathcal{S}_{\text{sin}}^{*}}(\mathcal{T}_{1})= 2 (\sin 1)/(3+2 \sin 1)\approx 0.35938$.
		\item $R_{\mathcal{S}_{\text{sin}}^{*}}(\mathcal{T}_{2})=(\sin 1)/3.$
	\end{enumerate}
\end{corollary}

\begin{proof}
	It was proved in \cite{viren} that $\{w:|w-a|< \sin 1- |a-1|\}\subseteq q(\mathbb{D})$ for $|a-1| \leq \sin 1$, where $q(z):=1+ \sin z$. For $a=1$, this implies that $\mathcal{S}^*_{1-\sin 1} \subset \mathcal{S}_{\text{sin}}^{*}$.
	 Thus $R_{\mathcal{S}^* _{1-\sin 1}}(\mathcal{T}_i)\leq R_{\mathcal{S}^*_{\text{sin}}}(\mathcal{T}_i) $ for $i=1,2$. The proof is completed by demonstrating $ R_{\mathcal{S}^*_{\text{sin}}}(\mathcal{T}_i)\leq R_{\mathcal{S}^* _{1-\sin 1}}(\mathcal{T}_i) $ for $i=1,2$.
	
	$(i)$ Evaluating the function $f_{1}(z)=z(1+z)^{3/2}$ at $z=-R=-R_{S^*_{1-\sin 1}}(\mathcal{T}_{1})=-2 \sin 1/(3+2 \sin 1)$ gives
	\[  \frac{z f'_{1}(z)}{f_{1}(z)}=\frac{2+5z}{2+2z}=\frac{2-5R}{2-2R}= 1-\sin 1=q(-1).\]
	Thus, $R_{\mathcal{S}_{\text{sin}}^{*}}(\mathcal{T}_{1})\leq 2 \sin 1/(3+2 \sin 1)$.
	
	$(ii)$ Similarly, at $z=\pm R=\pm R_{S^*_{1-\sin 1}}(\mathcal{T}_{2})=\pm(\sin 1)/3$, the function $f_{2}(z)=ze^{3z}$ yields
	\[  \frac{z f'_{2}(z)}{f_{2}(z)}=1+3z=1\pm 3R =1\pm \sin 1=q(\pm 1). \]
	This proves that  $R_{\mathcal{S}_{\text{sin}}^{*}}(\mathcal{T}_{2})\leq (\sin 1)/3.$
\end{proof}

Consider next the class $ \mathcal{S}_{\leftmoon}^{*}:=\mathcal{S}^*(z+\sqrt{1+z^{2}})$ introduced by Raina and Sok\'{o}\l{} in \cite{raina}. Functions $f\in \mathcal{S}_{\leftmoon}^{*}$ provided $zf'(z)/f(z)$ lies in the region bounded by the lune $\Omega_{l} :=\{w: |w^{2}-1|<2|w|\}$. The result below gives the radius of lune starlikeness for each class $\mathcal{T}_{1}$ and $\mathcal{T}_{2}$.

\begin{corollary}\label{p8thm3.6}
	The following are the $ \mathcal{S}_{\leftmoon}^{*}$-radius for each class $\mathcal{T}_{1}$ and $\mathcal{T}_{2}$:
	\begin{enumerate}[(i)]
		\item $R_{\mathcal{S}_{\leftmoon}^{*}}(\mathcal{T}_{1})= 2 (\sqrt{2}-2)/(2 \sqrt{2}-7)\approx 0.280847$.
		\item $R_{\mathcal{S}_{\leftmoon}^{*}}(\mathcal{T}_{2})=(2-\sqrt{2})/3.$
	\end{enumerate}
\end{corollary}

\begin{proof}
It was shown by Gandhi and Ravichandran \cite[Lemma 2.1]{gan} that $\{w:|w-a|<1-|\sqrt{2}-a|\}\subseteq \Omega_{l}$ for $\sqrt{2}-1<a\leq  \sqrt{2}+1$. Choosing $a=1$, the inclusion gives $\mathcal{S}^*_{\sqrt{2}-1}\subset \mathcal{S}_{\leftmoon}^{*}$. Thus $R_{\mathcal{S}^* _{\sqrt{2}-1}}(\mathcal{T}_i)\leq R_{\mathcal{S}^*_{\leftmoon}}(\mathcal{T}_i) $ for $i=1,2$. We complete the proof by demonstrating $ R_{\mathcal{S}^*_{\leftmoon}}(\mathcal{T}_i)\leq R_{\mathcal{S}^* _{\sqrt{2}-1}}(\mathcal{T}_i) $ for $i=1,2$.

$(i)$ Evaluating the function $f_{1}(z)=z(1+z)^{3/2}$ at $z=-R=-R_{S^*_{\sqrt{2}-1}}(\mathcal{T}_{1})=-2 (\sqrt{2}-2)/(2 \sqrt{2}-7)$ gives
\[  \left|\left(\frac{z f'_{1}(z)}{f_{1}(z)}\right)^{2}-1\right|=\left|\left(\frac{2+5z}{2+2z}\right)^{2}-1\right|=\left|\left(\frac{2-5R}{2-2R}\right)^{2}-1\right|= 0.828 = 2\left|\frac{z f'_{1}(z)}{f_{1}(z)}\right|. \]
Thus, $R_{S^*_{\leftmoon}}(\mathcal{T}_{1})\leq 2 (\sqrt{2}-2)/(2 \sqrt{2}-7)$.

$(ii)$ Similarly, at $z=-R=-R_{S^*_{\sqrt{2}-1}}(\mathcal{T}_{2})=-(2-\sqrt{2})/3$, the function $f_{2}(z)=ze^{3z}$ yields
\[  \left|\left(\frac{z f'_{2}(z)}{f_{2}(z)}\right)^{2}-1\right|=|(1+3z)^{2}-1|=|(1-3R)^{2}-1| =0.828=2\left|\frac{z f'_{2}(z)}{f_{2}(z)}\right|. \]
This proves that  $R_{S^*_{\leftmoon}}(\mathcal{T}_{2})\leq (2-\sqrt{2})/3.$
\end{proof}

As a further example, consider next the class $\mathcal{S}_{R}^{*}:=\mathcal{S}^*(\eta(z))$, where $\eta(z)=1+((zk+z^2)/(k^2-kz))$, $k=\sqrt{2}+1$. This class associated with a rational function was introduced and studied by Kumar and Ravichandran in \cite{sush}.

\begin{corollary}\label{p8thm3.7}
	The following are the $\mathcal{S}_{R}^{*}$-radius  for the classes $\mathcal{T}_{1}$ and $\mathcal{T}_{2}$:
	\begin{enumerate}[(i)]
		\item $R_{\mathcal{S}_{R}^{*}}(\mathcal{T}_{1})= 2(-3+2 \sqrt{2})/(4 \sqrt{2}-9)\approx 0.102642$,
		\item $R_{\mathcal{S}_{R}^{*}}(\mathcal{T}_{2})=(3-2\sqrt{2})/3.$
	\end{enumerate}
\end{corollary}

\begin{proof}
It was shown in \cite{sush} that $\{w:|w-a|<a-2(\sqrt{2}-1)\}\subseteq \eta(\mathbb{D})$ for $2(\sqrt{2}-1)<a\leq \sqrt{2}$. This inclusion with $a=1$ gives $\mathcal{S}^*_{2(\sqrt{2}-1)}\subset \mathcal{S}_{R}^{*}$. Thus	$R_{\mathcal{S}^* _{2(\sqrt{2}-1)}}(\mathcal{T}_i)\leq R_{\mathcal{S}_{R}^{*}}(\mathcal{T}_i) $ for $i=1,2$. We next show that $ R_{\mathcal{S}_{R}^{*}}(\mathcal{T}_i)\leq R_{\mathcal{S}^* _{2(\sqrt{2}-1)}}(\mathcal{T}_i) $ for $i=1,2$.
	
	$(i)$  At $z=-R=-R_{S^*_{2(\sqrt{2}-1)}}(\mathcal{T}_{1})=-2(-3+2 \sqrt{2})/(4 \sqrt{2}-9)$, the function $f_{1}(z)=z(1+z)^{3/2}$  yields
	 \[ \frac{z f'_{1}(z)}{f_{1}(z)}=\frac{2-5R}{2-2R} = 2(\sqrt{2}-1)=\eta(1).\]
	Thus, $R_{\mathcal{S}_{R}^{*}}(\mathcal{T}_{1})\leq 2(-3+2 \sqrt{2})/(4 \sqrt{2}-9)$.
	
	$(ii)$     Evaluating $f_{2}(z)=ze^{3z}$ at $z=-R=-R_{S^*_{2(\sqrt{2}-1)}}(\mathcal{T}_{2})=-(3-2\sqrt{2})/3$ gives
	\[  \frac{z f'_{2}(z)}{f_{2}(z)} = 1-3R =2(\sqrt{2}-1)=\eta(1).\]
	Thus $R_{\mathcal{S}_{R}^{*}}(\mathcal{T}_{2})\leq (3-2\sqrt{2})/3.$
\end{proof}

	%

The class $\mathcal{S}^{*}_{N_e}:=\mathcal{S}^*(\psi(z))$, where $\psi(z)=1+z-z^{3}/3$, was introduced and studied by Wani and Swaminathan in \cite{wani}. Geometrically, $f\in \mathcal{S}^{*}_{N_e}$ provided $zf'(z)/f(z)$ lies in the region bounded by the nephroid: a 2-cusped kidney shaped curve $\Omega_{N_e} :=\{w=u+iv: ((u-1)^{2}+v^{2}-4/9)^{3}-4v^{2}/3=0\}$.

\begin{corollary}\label{p8thm3.8}
	The following are the $\mathcal{S}^{*}_{N_e}$-radius for the classes $\mathcal{T}_{1}$ and $\mathcal{T}_{2}$:
	\begin{enumerate}[(i)]
		\item $R_{\mathcal{S}^{*}_{N_e}}(\mathcal{T}_{1})= 4/13$,
		\item $R_{\mathcal{S}^{*}_{N_e}}(\mathcal{T}_{2})=2/9.$
	\end{enumerate}
\end{corollary}

\begin{proof}
It was shown in \cite{wani} that $\{w:|w-a|<a-1/3\}\subseteq \Omega_{N_e}$ for $1/3<a\leq 1$. This inclusion with $a=1$ gives $\mathcal{S}^*_{1/3}\subset \mathcal{S}_{N_e}^{*}$. This shows that $R_{\mathcal{S}^* _{1/3}}(\mathcal{T}_i)\leq R_{\mathcal{S}^*_{N_e}}(\mathcal{T}_i) $ for $i=1,2$. We next show that $ R_{\mathcal{S}^*_{N_e}}(\mathcal{T}_i)\leq R_{\mathcal{S}^* _{1/3}}(\mathcal{T}_i) $ for $i=1,2$.
	
	 $(i)$  Evaluating the function $f_{1}(z)=z(1+z)^{3/2}$ at $z=-R=-R_{S^*_{1/3}}(\mathcal{T}_{1})=-4/13$ results in
	 \[   \frac{z f'_{1}(z)}{f_{1}(z)}=\frac{2-5R}{2-2R}= \frac{1}{3}=\psi(-1).\]
	 Thus, $R_{S^*_{N_e}}(\mathcal{T}_{1})\leq 4/13$.
	
	 $(ii)$ Similarly, evaluating $f_{2}(z)=ze^{3z}$ at $z=-R=-R_{S^*_{1/3}}(\mathcal{T}_{2})=-2/9$ yields
	\[  \frac{z f'_{2}(z)}{f_{2}(z)}=1-3R=\frac{1}{3}=\psi(-1).\]
	 This proves that  $R_{S^*_{N_e}}(\mathcal{T}_{2})\leq 2/9.$
\end{proof}


Finally, we consider the class $\mathcal{S}^{*}_{SG}:=\mathcal{S}^*(2/(1+e^{-z}))$ introduced by Goel and Kumar in \cite{goel}. Here $2/(1+e^{-z})$ is the modified sigmoid function that maps $\mathbb{D}$ onto the region $\Omega_{SG} :=\{w=u+iv: \left| \log (w/(2-w)) \right|<1\}$. Thus, $f\in \mathcal{S}^{*}_{SG}$ provided the function $zf'(z)/f(z)$ maps $\mathbb{D}$ onto the region lying inside the domain $\Omega_{SG}$.

\begin{corollary}\label{p8thm3.9}
	The $\mathcal{S}^{*}_{SG}$-radius for the class $\mathcal{T}_{1}$ is
\[R_{\mathcal{S}^{*}_{SG}}(\mathcal{T}_{1})= (2e-2)/(1+5e)\approx 0.23552,\]
while that of $\mathcal{T}_{2}$ is
\[R_{\mathcal{S}^{*}_{SG}}(\mathcal{T}_{2})=(e-1)/(3(1+e)).\]
\end{corollary}

\begin{proof}
The inclusion $\{w:|w-a|<((e-1)/(e+1))-|a-1|\}\subseteq \Omega_{SG}$ holds for $2/(1+e)<a<2e/(1+e)$ (see \cite{goel}). At $a=1$, the set inclusion shows that $\mathcal{S}^*_{2/(e+1)}\subset \mathcal{S}_{SG}^{*}$. It was also shown in \cite{goel} that $\mathcal{S}_{SG}^{*}\subset \mathcal{S}^* (\alpha)$ for $0 \leq \alpha \leq 2/(e+1)$. The desired result is now an immediate consequence of Theorem~\ref{p8thm3.1}.
\end{proof}

\section*{Acknowledgment} The first author gratefully acknowledge support from a USM research university grant 1001.PMATHS.8011101.

\end{document}